\documentclass[12pt]{amsart}
\usepackage{amsmath, amssymb, latexsym, amsthm,bm}
\usepackage{mathrsfs}
\usepackage{fullpage}
\usepackage{color}
\usepackage{tikz}
\usepackage{graphicx}
\usepackage{stmaryrd,amsbsy,enumerate}
\usepackage{hyperref}

\usepackage{mathscinet}
\usetikzlibrary{calc}
\usepackage{cancel}
\usepackage{xspace}
\usepackage{url}
\usepackage{longtable}
\usepackage[obeyFinal,colorinlistoftodos,textsize=tiny]{todonotes}

\usepackage{tikz}

\usepackage{caption}

\usepackage{mathtools,amsmath}
\usepackage{array}

\usepackage{tikz}

\usepackage{stmaryrd}

\def\ignore#1{{ }}

\def\pp{{\mathscr P}}

\def\LL{{\mathscr L}}

\definecolor{mygray}{gray}{0.8}

\def\sphere{{\mathbb S}}

\def\mas#1{{{#1}^+}}
\def\menos#1{{{#1}^-}}

\def\proj{{\mathbb {R}\text{P}}^2}

\def\qq{\mathscr Q}
\def\nn{\mathscr N}

\def\aa{{\mathscr A}}

\def\mm{{\mathscr M}}
\def\nn{{\mathscr N}}

\def\cc{{\mathscr C}}

\def\xx{{\mathscr X}}

\def\myfrac#1#2{{\genfrac{}{}{0pt}{}{#1}{#2}}}

\def\cir#1#2#3{{\cc{\hbox{\hglue #1 cm}\myfrac{#2}{#3}}}}

\def\lonem{\line{-0.1}{1}{m}}
\def\lonem{{\mathscr A}_m}

\def\centerarc[#1](#2)(#3:#4:#5){ \draw[#1] ($(#2)+({#5*cos(#3)},{#5*sin(#3)})$) arc (#3:#4:#5); } 

\def\lonefive{\aa_{5}}

\def\conem{\cir{-0.1}{1}{m}}

\def\ctwom{\cir{-0.06}{2}{m}}

\def\ctwothree{\cir{-0.03}{2}{3}}
\def\cthreethree{\cir{-0.03}{3}{3}}
\def\conethree{\cir{-0.04}{1}{3}}

\def\conen{\cir{-0.06}{1}{n}}
\def\ctwon{\cir{-0.06}{2}{n}}
\def\cthreem{\cir{-0.08}{3}{m}}
\def\cthreen{\cir{-0.03}{3}{n}}

\def\conefive{\cir{-0.03}{1}{5}}
\def\ctwofive{\cir{-0.03}{2}{5}}
\def\cthreefive{\cir{-0.02}{3}{5}}

\def\c13{{\conethree}}
\def\c35{{\cthreefive}}

\newtheorem{theorem}{Theorem} 
\newtheorem{theorem*}{Theorem} 
\newtheorem{proposition}[theorem]{Proposition} 

\newtheorem{lemma}[theorem]{Lemma}

\newtheorem{observation}[theorem]{Observation}

\theoremstyle{definition}

\begin{document}



\title{\bf The unavoidable arrangements of pseudocircles}

\author{Carolina Medina}
\address{Instituto de F\'\i sica, Universidad Aut\'onoma de San Luis Potos\'{\i}, Mexico.}
\email{cmedina@ifisica.uaslp.mx}

\author{Jorge Ram\'\i rez-Alfons\'{\i}n}
\address{IMAG, Univ. Montpellier, CNRS, Montpellier, France.}
\email{jorge.ramirez-alfonsin@univ-montp2.fr}

\author{Gelasio Salazar}
\address{Instituto de F\'\i sica, Universidad Aut\'onoma de San Luis Potos\'{\i}, Mexico.} 
\email{gsalazar@ifisica.uaslp.mx}






\begin{abstract}
It is known that cyclic arrangements are the only {\em unavoidable} simple arrangements of pseudolines:  for each fixed $m\ge 1$, every sufficiently large simple arrangement of pseudolines has a cyclic subarrangement of size $m$. In the same spirit, we show that there are three unavoidable arrangements of pseudocircles.
\end{abstract}

\maketitle

\section{Introduction}\label{sec:introduction}

A  {\em pseudoline} is a noncontractible simple closed curve in the projective plane $\proj$. An {\em arrangement of pseudolines} is a set of pseudolines that cross each other exactly once. Two arrangements of pseudolines are {\em isomorphic} if the cell complexes they induce in $\proj$ are isomorphic. An  arrangement of pseudolines is {\em simple} if no three pseudolines have a common point. A simple arrangement is {\em cyclic} if its pseudolines can be labelled $1, 2,\ldots,m$, so that each pseudoline $i\in[m]$ intersects the pseudolines in $\{1,2,\ldots,m\}{\setminus}\{i\}$ in increasing order, as in Figure~\ref{fig:fig1}. We use $\lonem$ to denote a cyclic arrangement of size $m$. 

\begin{figure}[ht!]
\centering
\scalebox{0.7}{\begin{tikzpicture}[]

\begin{scope}[shift ={(0,0)}, ]

\draw[black] (0,2.5) .. controls (0,0) and (0,0) .. (0,-2.5);
\draw[blue,rotate =-18] (0,2.5) .. controls (-0.5,1) and (-0.5,-1) .. (0,-2.5);
(0,-4);
\draw[red,rotate =-36] (0,2.5) .. controls (0,0.5) and (-0.5,-1) .. (0,-2.5);
(0,-4);
\draw[green!60!black,rotate =-54] (0,2.5) .. controls (0.5,0.5) and (0.7,0.5) .. (0,-2.5);
\draw[magenta,rotate =-72] (0,2.5) .. controls (0.4,2) and (2.2,-0.5) .. (0,-2.5);
\draw[lightgray,dashed,thick ] (0,0) circle (2.5);
  \draw (0,2.8) node {\small $1$};
  \draw[rotate=-18,ultra thick] (0,2.8) node {\small $2$};
  \draw[rotate=-36, ultra thick] (0,2.8) node {\small $3$};
  \draw[rotate=-54, ultra thick] (0,2.8) node {\small $4$};
  \draw[rotate=-72, ultra thick] (0,2.8) node {\small $5$};
  \draw [rotate=180, ultra thick](0,2.8) node {\small $1$};
  \draw[rotate=162, ultra thick] (0,2.8) node {\small $2$};
  \draw[rotate=144, ultra thick] (0,2.8) node {\small $3$};
  \draw[rotate=126, ultra thick] (0,2.8) node {\small $4$};
  \draw[rotate=108, ultra thick] (0,2.8) node {\small $5$};
 \end{scope}
\end{tikzpicture}}
\caption{{The cyclic arrangement $\lonefive$.}}
\label{fig:fig1}
\end{figure}


In the spirit of~\cite{Pach1}, the following states that cyclic arrangements are the only {\em unavoidable} arrangements of pseudolines. 

\begin{theorem}[{\cite[Proposition 3.4.7]{Negami1}}]\label{thm:proj0}
For each fixed $m\ge 1$, every sufficiently large simple arrangement of pseudolines has a subarrangement isomorphic to $\lonem$.
\end{theorem}

We prove an analogue of Theorem~\ref{thm:proj0} for arrangements of pseudocircles. A {\em pseudocircle} is a simple closed curve in the sphere $\sphere^2$. We use Gr\"unbaum's original notion that an {\em arrangement of pseudocircles} is a set of pseudocircles that pairwise intersect at exactly two points, at which they cross, and no three pseudocircles have a common point~\cite{grunbaum}. This notion is still adopted nowadays~\cite{ortner2}, and some more general notions are also used in the literature~\cite{Felsner1,kang}. 


In Figure~\ref{fig:fig2} we illustrate arrangements $\conefive,\ctwofive$, $\cthreefive$, and it is clear how to generalize them to $\conem,\ctwom$, and $\cthreem$, for any $m\ge 1$. These are the unavoidable arrangements of pseudocircles.

\begin{theorem}\label{thm:main}
For each fixed $m\ge 1$, every sufficiently large arrangement of pseudocircles has a subarrangement isomorphic to $\conem,\ctwom$, or $\cthreem$.
\end{theorem}

\begin{figure}[ht!]
\centering
\scalebox{0.7}{\begin{tikzpicture}[]
\begin{scope}[shift ={(0,0)}, scale=1]
\draw[black] (0,1.) circle (1.5);
\draw[magenta,rotate =72] (0,1.) circle (1.5);
\draw[green!60!black,rotate =144] (0,1.) circle (1.5);
\draw[blue,rotate =-72] (0,1.) circle (1.5);
\draw[red,rotate =-144] (0,1.) circle (1.5);
\end{scope}
\begin{scope}[shift ={(7,0)},scale=1]
\draw[black] (-1.5,0) circle (2);
\draw[blue] (-0.75,0) circle (2);
\draw[red,] (0,0) circle (2);
\draw[green!60!black] (0.75,0) circle (2);
\draw[magenta] (1.5,0) circle (2);
\end{scope}
\begin{scope}[shift ={(17.2,-2.2)},scale=0.7]
\draw[magenta, rounded corners=2] (-4,0) -- (5,0)--(5,5.5)--(3.5,5.5)--(3.5,4)--(-8,4)--(-8,0)--(-4,0);
\draw[green!60!black, rounded corners=2] (-4,0.3) -- (3,0.3)--(3,5.5)--(1.5,5.5)--(1.5,3.7)--(-7.7,3.7)--(-7.7,0.3)--(-4,0.3);
\draw[red, rounded corners=2] (-4,0.6) -- (1,0.6)--(1,5.5)--(-0.5,5.5)--(-0.5,3.4)--(-7.4,3.4)--(-7.4,0.6)--(-4,0.6);
\draw[blue, rounded corners=2] (-4,0.9) -- (-1,0.9)--(-1,5.5)--(-2.5,5.5)--(-2.5,3.1)--(-7.1,3.1)--(-7.1,0.9)--(-4,0.9);
\draw[black, rounded corners=2] (-4,1.2) -- (-3,1.2)--(-3,5.5)--(-4.5,5.5)--(-4.5,2.8)--(-6.8,2.8)--(-6.8,1.2)--(-4,1.2);
  \draw (-24.5,7.2) node {\small $1$};
  \draw (-21,4.7) node {\small $2$};
  \draw (-21.6,0.8) node {\small $3$};
  \draw (-27.7,0.8) node {\small $4$};
  \draw (-28.2,4.6) node {\small $5$};

  \draw (-20,3) node {\small $1$};
  \draw (-18.8,3) node {\small $2$};
  \draw (-17.7,3) node {\small $3$};
  \draw (-16.6,3) node {\small $4$};
  \draw (-15.5,3) node {\small $5$};

  \draw (-3.3,2) node {\small $1$};
  \draw (-1.3,2) node {\small $2$};
  \draw (0.7,2) node {\small $3$};
  \draw (2.7,2) node {\small $4$};
  \draw (4.7,2) node {\small $5$};

\end{scope}

\end{tikzpicture}} 
\caption{The arrangements $\conefive$ (left), $\ctwofive$ (center), and $\cthreefive$ (right).}
\label{fig:fig2}
\end{figure}



We note that Theorem~\ref{thm:main} is best possible, as no pseudocircle can be added to the collection $\{\conem,\ctwom,\cthreem\}$. This follows since for all integers $m,n$ with $m\le n$, all subarrangements of $\conen$ (respectively, $\ctwon,\cthreen$) of size $m$ are isomorphic to $\conem$ (respectively, $\ctwom,\cthreem$). Thus it is accurate to say that these are {\em the} unavoidable arrangements of pseudocircles.

For the rest of the paper, for brevity we refer to an arrangement of pseudocircles simply as an {\em arrangement}. 

\section{Reducing Theorem~\ref{thm:main} to two kinds of arrangements}

There are, up to isomorphism, only two arrangements of size $3$. Following~\cite{F2}, these are the {\em Krupp arrangement} and the {\em NonKrupp arrangement}. We refer the reader to Figure~\ref{fig:fig6}. Note that the Krupp arrangement is isomorphic to $\conethree$, and the NonKrupp arrangement is isomorphic to $\ctwothree$ and $\cthreethree$. If an arrangement $\pp$ of size at least $3$ has all its $3$-subarrangements isomorphic to the Krupp arrangement (respectively, to the NonKrupp arrangement), then we say that $\pp$ is {\em Krupp-packed} (respectively, {\em NonKrupp-packed}).

\begin{figure}[ht!]
\centering
\scalebox{0.8}{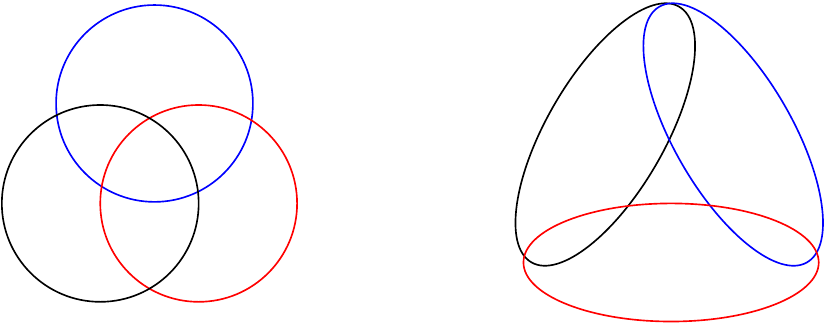}
\caption{{The Krupp arrangement (left) and the NonKrupp arrangement (right).}}
\label{fig:fig6}
\end{figure}

We now state Theorem~\ref{thm:main} for Krupp-packed and for NonKrupp-packed arrangements. As we shall see shortly, the general version of Theorem~\ref{thm:main} easily follows as a consequence.

\begin{lemma}\label{lem:krupp}
Theorem~\ref{thm:main} holds for Krupp-packed arrangements.
\end{lemma}

\begin{lemma}\label{lem:nonkrupp}
Theorem~\ref{thm:main} holds for NonKrupp-packed arrangements.
\end{lemma}

We use Ramsey theory in our arguments. We recall that the {\em order} of a hypergraph is its number of vertices, and $r_k(\ell_1,\ell_2,\ldots,\ell_n)$ denotes the Ramsey number for complete $k$-uniform hypergraphs. That is, if each $k$-edge of a complete $k$-uniform hypergraph of order at least $r_k(\ell_1,\ell_2,\ldots,\ell_n)$ has colour $i$ for some $i\in[n]$, then there is an $i\in[n]$ and a complete subhypergraph of order $\ell_i$, all of whose $k$-edges have colour $i$. 

\begin{proof}[Proof of Theorem~\ref{thm:main}, assuming Lemmas~\ref{lem:krupp} and~\ref{lem:nonkrupp}]
Let $m\ge 1$ be a fixed integer. Assuming Lemmas~\ref{lem:krupp} and~\ref{lem:nonkrupp}, it follows that there is an integer $p$ such that every Krupp-packed or NonKrupp-packed arrangement of size at least $p$ has a subarrangement isomorphic to $\conem,\ctwom$, or $\cthreem$. Let $\qq$ be an arrangement of size $q=r_3(p,p)$. Regard $\qq$ as a complete $3${-}uniform hypergraph, and colour a $3${-}edge blue (respectively, red) if the pseudocircles in the $3$-edge form an arrangement isomorphic to the Krupp arrangement (respectively, to the NonKrupp arrangement). 

Since $q=r_3(p,p)$ it follows from Ramsey's theorem that $\qq$ has a subarrangement $\pp$ of size $p$ that is either Krupp-packed or NonKrupp-packed. The assumption on $p$ implies that $\pp$, and hence $\qq$, contains a subarrangement isomorphic to $\conem,\ctwom$, or $\cthreem$.
\end{proof}

We finish this section by proving Lemma~\ref{lem:krupp}. The rest of the paper is devoted to the proof of Lemma~\ref{lem:nonkrupp}.

\begin{proof}[Proof of Lemma~\ref{lem:krupp}]
The key fact we use here is the following. Every Krupp{-}packed arrangement $\pp$ (known in the literature as  an {\em arrangement of great pseudocircles}) can be obtained from a simple arrangement $\LL$ of pseudolines by suitably gluing together two wiring diagrams of $\LL$, as illustrated in Figure~\ref{fig:fig4} (see~\cite[Section 3.2]{Felsner1} and~\cite[Section 6.1.4]{richterziegler}). For a discussion on the wiring diagram representation of a pseudoline arrangement we refer the reader to~\cite{f0}. 

If $\LL$ is a cyclic arrangement $\aa_m$, it is readily seen that $\pp$ is isomorphic to $\conem$.

\begin{figure}[ht!]
\centering
\hglue -0.4 cm\scalebox{0.42}{\begin{tikzpicture}[]
%
\begin{scope}[shift ={(0,0)}]
\draw[black] (0,1)..  controls (0.25,1)  and (3.75,-3) .. (4,-3)--(7,-3);
\draw[blue] (0,0)..  controls (0.25,0) and (0.75,1) .. (1,1)-- (2,1)..  controls (2.25,1) and (4.75,-2) .. (5,-2)--(7,-2);
\draw[red] (0,-1)-- (1,-1)..  controls (1.5,-1) and (2.5,1) .. (3,1)--(4,1)..  controls (4.5,1)  and (5.5,-1) .. (6,-1)--(7,-1);
\draw[green!60!black] (0,-2)--(2,-2)..  controls (2.5,-2) and (4.5,1) .. (5,1)-- (6,1)..  controls (6.5,1) and (6.5,0) .. (7,0);
\draw[magenta](0,-3)-- (3,-3)..  controls (3.5,-3) and (6.5,1) .. (7,1);
\end{scope}
\begin{scope}[shift ={(7,-2)},rotate=0]
\draw[black] (0,-1)..  controls (0.25,-1)  and (3.75,3) .. (4,3)--(7,3);
\draw[blue] (0,0)..  controls (0.25,0) and (0.75,-1) .. (1,-1)-- (2,-1)..  controls (2.25,-1) and (4.75,2) .. (5,2)--(7,2);
\draw[red] (0,1)-- (1,1)..  controls (1.5,1) and (2.5,-1) .. (3,-1)--(4,-1)..  controls (4.5,-1)  and (5.5,1) .. (6,1)--(7,1);
\draw[green!60!black] (0,2)--(2,2)..  controls (2.5,2) and (4.5,-1) .. (5,-1)-- (6,-1)..  controls (6.5,-1) and (6.5,0) .. (7,0);
\draw[magenta](0,3)-- (3,3)..  controls (3.5,3) and (6.5,-1) .. (7,-1);
\draw[dashed,gray] (-15,-2) rectangle (0.15,7.0);
\draw[dashed,gray] (0.15,-2) rectangle (15.3,7.0);
\end{scope}

\begin{scope}

\draw[black] (6,2) .. controls (-4,2) and (-1,1).. (0,1);
\draw[black] (6,2) .. controls (18,2) and (15,1).. (14,1);
\draw[blue] (6,2.8) .. controls (21,2.8) and (16,0).. (14,0);
\draw[blue] (6,2.8) .. controls (-6,2.8) and (-2,0).. (0,0);
\draw[red] (6,3.3) .. controls (-9,3.3) and (-3,-1).. (0,-1);
\draw[red] (6,3.3) .. controls (24,3.3) and (17,-1).. (14,-1);
\draw[green!60!black] (6,3.7) .. controls (-11,3.7) and (-4,-2).. (0,-2);
\draw[green!60!black] (6,3.7) .. controls (27,3.7) and (18,-2).. (14,-2);
\draw[magenta] (6,4) .. controls (-14,4) and (-5,-3).. (0,-3);
\draw[magenta] (6,4) .. controls (31,4) and (19,-3).. (14,-3);
\end{scope}

\end{tikzpicture}}
\caption{{Obtaining $\conefive$ by suitably gluing together two wiring diagrams of ${\mathscr A}_5$.}}
\label{fig:fig4}
\end{figure}

Let $m\ge 1$ be a fixed integer. Let $\pp$ be a Krupp-packed arrangement of size $p:=r_3(m,m)$. Let $\LL$ be a simple arrangement of pseudolines that induces $\pp$, in the sense of the previous paragraph. Since $p=r_3(m,m)$, then by~\cite[Proposition 1.4]{Alfonsin1} $\LL$ has a cyclic arrangement $\lonem$ as a subarrangement. The subarrangement of $\pp$ induced by the pseudolines in $\lonem$ is obtained by suitably gluing together two copies of $\lonem$, and so it is isomorphic to $\conem$.
\end{proof}

\section{Intersection codes}\label{sec:codes}

In the proof of Lemma~\ref{lem:nonkrupp} we use intersection codes, as developed in~\cite{linor1,ortner2}. {This framework, in which one naturally encodes combinatorially essential information of an arrangement, can be seen as a generalization of the axiomatization of oriented matroids based on hyperline sequences~\cite{bokowski1}, and its essence goes back to the work of Gauss on planar curves in the 1830s.} 


Let $\pp=\{1,\ldots,n\}$ be an arrangement. For each $i\in\pp$ choose a point $p_i$ not contained in any other pseudocircle, and also choose one of the two possible orientations for $i$, so that for each pseudocircle we can naturally speak of a left side and a right side. Suppose that as we traverse $i$ starting at $p_i$ following the chosen orientation, we intersect $j$ from the left (respectively, right) side of $j$. We record this by writing $j^+$ (respectively, $j^-$). By keeping track of the order in which the intersections occur, we obtain the {\em code} of $i$ in $\pp$. Thus the code of each $i$ is a permutation of $\bigcup_{j\in[n]\setminus\{i\}}\{j^+,j^-\}$. If we omit the superscripts $+$ and $-$, we obtain the {\em unsigned} code of $i$. 

For instance, suppose that $\pp$ is $\ctwofive$ in Figure~\ref{fig:fig2}. Choose the counterclockwise orientation for each pseudocircle. If we choose as starting point for $3$ its leftmost point, then its code is $\mas{1}\mas{2}\menos{4}\menos{5}\mas{5}\mas{4}\menos{2}\menos{1}$, and its unsigned code is $12455421$.


We make essential use of the following.

\begin{proposition}[{\cite[Section 3]{linor1},\cite[Section 2]{ortner2}}]\label{pro:ortner}
Let $\pp,\qq$ be arrangements, both of which have their pseudocircles labelled $1,2,\ldots,n$, where an orientation and an initial traversal point has been chosen for each pseudocircle in $\pp$ and each pseudocircle in $\qq$. Suppose that for each $i\in[n]$, the code of $i$ in $\pp$ is the same as the code of $i$ in $\qq$. Then $\pp$ and $\qq$ are isomorphic.
\end{proposition}

We now introduce some notation to describe the codes of $\ctwom$ and $\cthreem$ in a compact manner. Let $i,m$ be integers such that $1\le i \le m$. We use $[\mas{1}{:}\mas{i})$ to denote the string $\mas{1}\mas{2}\cdots \mas{(i{{-}}1)}$. In a similar spirit, we use $(\menos{i}{:}\menos{1}]$ to denote $\menos{(i-1)}\menos{(i-2)} \cdots \menos{1}$; we use $(\menos{i}{:}\menos{m}]$ to denote $\menos{(i+1)}\menos{(i+2)}\cdots \menos{m}$; and we use $[\mas{m}{:}\mas{i})$ to denote $\mas{m}\cdots \mas{(i+2)}\mas{(i+1)}$. Finally, we use $[\mas{1}\menos{1}{:}\mas{i}\menos{i})$ to denote $\mas{1}\menos{1}\mas{2}\menos{2}\cdots \mas{(i{-}1)} \menos{(i{-}1)}$. Note that $[\mas{1}{:}\mas{1}), (\menos{1}{:}\menos{1}]$, and $[\mas{1}\menos{1}{:}\mas{1}\menos{1})$ are empty strings, and $(\menos{m}{:}\menos{m}]$ and $[\mas{m}{:}\mas{m})$ are also empty strings. With this notation, we have the following.

\begin{observation}\label{obs:obsctwo}
Label $\ctwom$ with the natural extension of the labelling of $\ctwofive$ in Figure~\ref{fig:fig2}. Orient all pseudocircles counterclockwise, and for each pseudocircle choose as initial traversal point its leftmost point. Then the code of each $i$ in $\ctwom$ is $ [\mas{1}{:}\mas{i})  (\menos{i}{:}\menos{m}] [\mas{m}{:}\mas{i})  (\menos{i}{:}\menos{1}]$.
\end{observation}

\begin{observation}\label{obs:obscthree}
Label $\cthreem$ with the natural extension of the labelling of $\cthreefive$ in Figure~\ref{fig:fig2}. Orient all pseudocircles clockwise, and for each pseudocircle choose as initial traversal point its bottom right corner. Then the code of each $i$ in $\cthreem$ is $[\mas{1}\menos{1}{:}\mas{i}\menos{i}) (\menos{i}{:}\menos{m}][\mas{m}{:}\mas{i})$.
\end{observation}

We close this section with a remark on NonKrupp-packed arrangements. We say that an arrangement is {\em bad} if its pseudocircles can be labelled $1,\ldots,n$ so that for each pseudocircle $i=1,2,\ldots,n-2$, the unsigned code of $i$ in the subarrangement $\{i,i+1,\ldots,n\}$ is $(i+1)(i+1)(i+2)(i+2)\cdots n n$. Up to isomorphism, there is only one bad NonKrupp-packed arrangement of size $4$, namely the arrangement $\xx_4$ shown in Figure~\ref{fig:fig5}. This is easily checked by hand, or by an inspection of~\cite[Figure 2]{Felsner1}, which contains all arrangements of size $4$. 

{
\begin{figure}[ht!]
\centering
\scalebox{0.6}{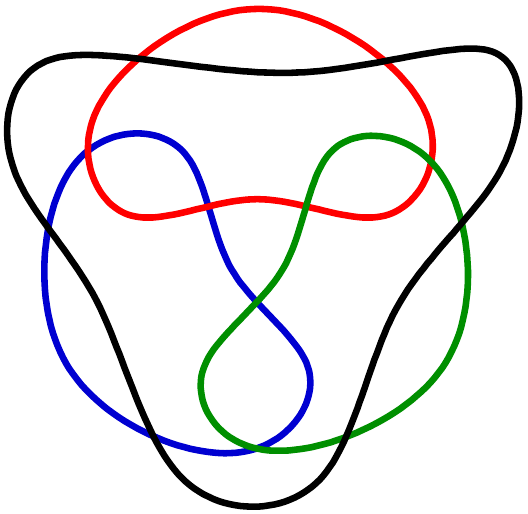}
\caption{The arrangement $\xx_4$.} 
\label{fig:fig5}
\end{figure}
}

In a bad NonKrupp-packed arrangement of size $5$, all $4$-subarrangements would then be isomorphic to $\xx_4$. A routine case analysis by hand shows that no arrangement of size $5$ satisfies this property. We highlight this remark, as we use it in the proof of Lemma~\ref{lem:nonkrupp}.

\begin{observation}\label{obs:nox4packed}
There is no bad NonKrupp-packed arrangement of size $5$ (or larger).
\end{observation}

\section{Proof of Lemma~\ref{lem:nonkrupp}}

First we identify a property shared by $\ctwom$ and $\cthreem$. To motivate this, we refer the reader to $\ctwofive$ and $\cthreefive$, shown in Figure~\ref{fig:fig2}. If we perform the relabelling $i\mapsto i-1$ to the pseudocircles in either of these arrangements,  the parts of $1,2,3,4$ in one component of $\sphere^2\setminus\{0\}$ are pairwise disjoint, and appear in a rainbow-like fashion: the unsigned code of $0$ is $12344321$. We say that an ar\-range\-ment is {\em rainbow} if its pseudocircles can be labelled $0,1,\ldots,n$ so that (I) one of the components of $\sphere^2\setminus\{0\}$ contains no intersections among the pseudocircles $1,2,\ldots,n$; and (II) the unsigned code of pseudocircle $0$ is $12\cdots n n\cdots 2 1$.


Lemma~\ref{lem:nonkrupp} follows easily from the next two statements. 

\begin{proposition}\label{pro:claima} 
For each fixed integer $n\ge 5$, every sufficiently large NonKrupp{-}packed arrangement has a rainbow subarrangement of size $n$.
\end{proposition}

\begin{proposition}\label{pro:claimb}
For each fixed integer $m\ge 5$, every sufficiently large NonKrupp-packed rainbow arrangement contains a subarrangement isomorphic to $\ctwom$ or $\cthreem$.
\end{proposition}

Before proving these propositions, for completeness we give the proof of Lemma~\ref{lem:nonkrupp}.

\begin{proof}[Proof of Lemma~\ref{lem:nonkrupp}, assuming Propositions~\ref{pro:claima} and~\ref{pro:claimb}] 
Obviously it suffices to prove Theorem~\ref{thm:main} for every integer $m\ge 5$. Let $m\ge 5$ be a fixed integer. By Proposition~\ref{pro:claimb}, there is an integer $n:=n(m)$ such that every NonKrupp-packed rainbow arrangement contains a subarrangement isomorphic to $\ctwom$ or $\cthreem$. By Proposition~\ref{pro:claima}, there is an integer $q:=q(n)$ such that every NonKrupp-packed arrangement has a rainbow subarrangement of size $n$. Thus every NonKrupp-packed arrangement of size at least $q$ contains a subarrangement isomorphic to $\ctwom$ or $\cthreem$.
\end{proof}

\begin{proof}[Proof of Proposition~\ref{pro:claima}] Let $n\,{\ge}\,5$ be a fixed integer. Let $p:=r_3(n,n)+1$ and $q:=r_3(p,p,p,p)$. We let $\qq=\{1,\ldots,q\}$ be a NonKrupp-packed arrangement, and show that $\qq$ contains a rainbow subarrangement of size $n$. 

Choose an arbitrary starting traversal point and orientation for each pseudocircle in $\qq$. The NonKrupp-packedness of $\qq$ implies that if $j,k,\ell$ are pseudocircles in $\qq$ such that $j<k<\ell$, then the unsigned code of $j$ in the subarrangement $\{j,k,\ell\}$ is either (i) $kk\ell\ell$; or (ii) $\ell\ell k k$; or (iii) $k\ell\ell k$; or (iv) $\ell k k \ell$. 

Regard $\qq$ as a complete $3${-}uniform hypergraph, and assign to each $3${-}edge $\{j,k,\ell\}$ with $j{<}k{<}\ell$ the colour (i), (ii), (iii), or (iv), depending on which of these scenarios holds. By Ramsey's theorem, $\qq$ has a subarrangement $\pp=\{1',2',\ldots,p'\}$, with $1\le 1' < \cdots < p' \le q$, all of whose $3$-edges are of the same colour. 

Suppose that all $3$-edges of $\pp$ are of colour (i). Then for each $i=1,\ldots,p-2$, the unsigned code of $i'$ in the subarrangement $\{i',{(i+1)}',\ldots,p'\}$ is $(i+1)'(i+1)'\cdots p' p'$. Thus $\pp$ is a bad arrangement of size $p>5$, contradicting Observation~\ref{obs:nox4packed}. Thus not all $3$-edges of $\pp$ can be of colour (i). An analogous argument shows that not all $3$-edges can be of colour (ii). 

If all $3$-edges of $\pp$ are of colour (iv) then by relabelling the pseudocircles in the reverse order we obtain an arrangement in which all $3$-edges are of colour (iii). Thus we may assume that all $3$-edges of $\pp$ are of colour (iii). In particular, the unsigned code of $1'$ in $\{1',\ldots,p'\}$ is $2' 3' \cdots p' p' \cdots 3' 2'$. Using that $p-1=r_2(n,n)$, an application of Ramsey's theorem shows that there exist $i_1',i_2',\ldots,i_n'$, with $2' \le i_1' < i_2' \cdots < i_n'\le p'$ such that one of the connected components of $\sphere^2\setminus\{1'\}$ contains no intersections among the pseudocircles $i_1',\ldots,i_n'$. The arrangement $\{1',i_1',\ldots,i_n'\}$ is rainbow. To see this, it suffices to relabel $1'$ with $0$, and $i_j'$ with $j$, for each $j=1,\ldots,n$.
\end{proof}



\begin{proof}[Proof of Proposition~\ref{pro:claimb}]
Let $m\ge5$ be a fixed integer. Let $q=r_3(m,m,m)$, and $n=r_3(q,q)$. Let $\nn_0=\{0,1,2,\ldots,n\}$ be a NonKrupp-packed rainbow arran\-gement. We show that $\nn_0$ contains a subarrangement isomorphic to either $\ctwom$ or to $\cthreem$.

Our first goal is to show that we may assume that the layout of $\nn_0$ is as shown on the right hand side of Figure~\ref{fig:figV}. To achieve this, first we note that by performing a self-homeomorphism of the sphere we may assume that the pseudocircle $0$ is the union of the Greenwich Meridian and the 180th Meridian, in particular passing through the north pole $N$ and the south pole $S$. We orient $0$ in the direction from S to N following the Greenwich Meridian, as on the left hand side of Figure~\ref{fig:figV}. 


\begin{figure}[ht!]
\centering
\scalebox{0.498}{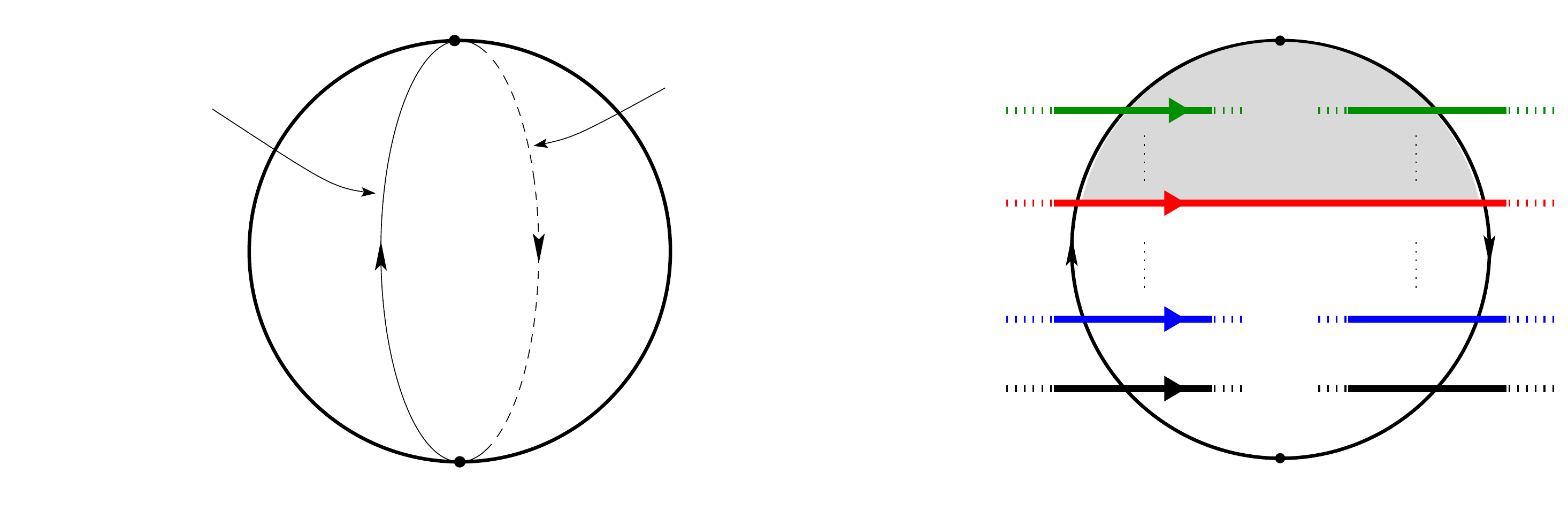}
\caption{Initial setup in the proof of Proposition~\ref{pro:claimb}. On the left hand side we illustrate pseudocircle $0$. On the right hand side we illustrate the eastern hemi\-sphere, which contains all intersections among the pseudocircles $1,\ldots,n$. We illustrate the regions $N_k$ (grey) and $S_k$ (white) of pseudocircle $k$.}
\label{fig:figV}
\end{figure}

Since $\nn_0$ satisfies rainbowness Property (II), we may assume that as we traverse the Greenwich Meridian from S to N we intersect the pseudocircles $1,2,\ldots,n$ in this order, and as we traverse the 180th Meridian from N to S, we intersect them in the order $n,\ldots,2,1$. Since $\nn_0$ satisfies rainbowness Property (I), then either the eastern or the western hemisphere (say the eastern one) contains all the intersections among the pseudocircles in $\nn:=\{1,2,\ldots,n\}$. We orient all pseudocircles in $\nn$ so that as we traverse $0$ along the Greenwich Meridian from S to N the code we obtain is $\menos{1}\menos{2}\cdots\menos{n}$ (that is, these pseudocircles hit $0$ from its left hand side). Thus as we traverse the 180th Meridian from N to S the code is $\mas{n}\mas{(n-1)}\cdots \mas{1}$ . Each pseudocircle $k\in\nn$ decomposes the eastern hemisphere into two parts, a part $N_k$ that contains N, and a part $S_k$ that contains $S$. Thus the setup is as illustrated in Figure~\ref{fig:figV}.


Since $\nn$ is NonKrupp{-}packed, for any three pseudocircles $j,k,\ell$ in $\nn$, either both intersections of $j$ and $\ell$ occur in $N_k$, or they both occur in $S_k$. Regard $\nn$ as a complete $3${-}uniform hypergraph, and assign to a $3${-}edge $\{j,k,\ell\}$ with $j{<}k{<}\ell$ the colour $N$ (respectively, $S$) if the intersections of $j$ and $\ell$ lie on $N_k$ (respectively, $S_k$). Since $|\nn|=n=r_3(q,q)$, then by Ramsey's theorem $\nn$ has a subarrangement $\qq=\{1',2',\ldots,q'\}$ of size $q$, with $1'{<}2'{<}\cdots{<}q'$, all of whose $3$-edges are of the same colour. By symmetry we may assume that all $3$-edges of $\qq$ have colour $N$. To avoid unnecessary cluttered notation, we relabel the pseudocircles in $\qq$ as $1,2,\ldots,q$.

Let $j,k,\ell$ be such that $1\le j {<} k {<} \ell\le q$. Since the intersections between $j$ and $\ell$ occur in $N_k$, that is, above $k$, and $\qq$ is NonKrupp-packed, then there are only three possibilities for how $j$ and $\ell$ can intersect each other, namely as shown in Figure~\ref{fig:figW}(a), (b), and (c). Thus the code of $k$ in $\{j,k,\ell\}$ is either (1) $\mas{j}\menos{\ell}\mas{\ell}\menos{j}$, or (2) $\mas{j}\menos{j}\menos{\ell}\mas\ell$, or (3) $\menos{\ell}\mas{\ell}\mas{j}\menos{j}$, respectively.

\begin{figure}[ht!]
\centering
\scalebox{0.45}{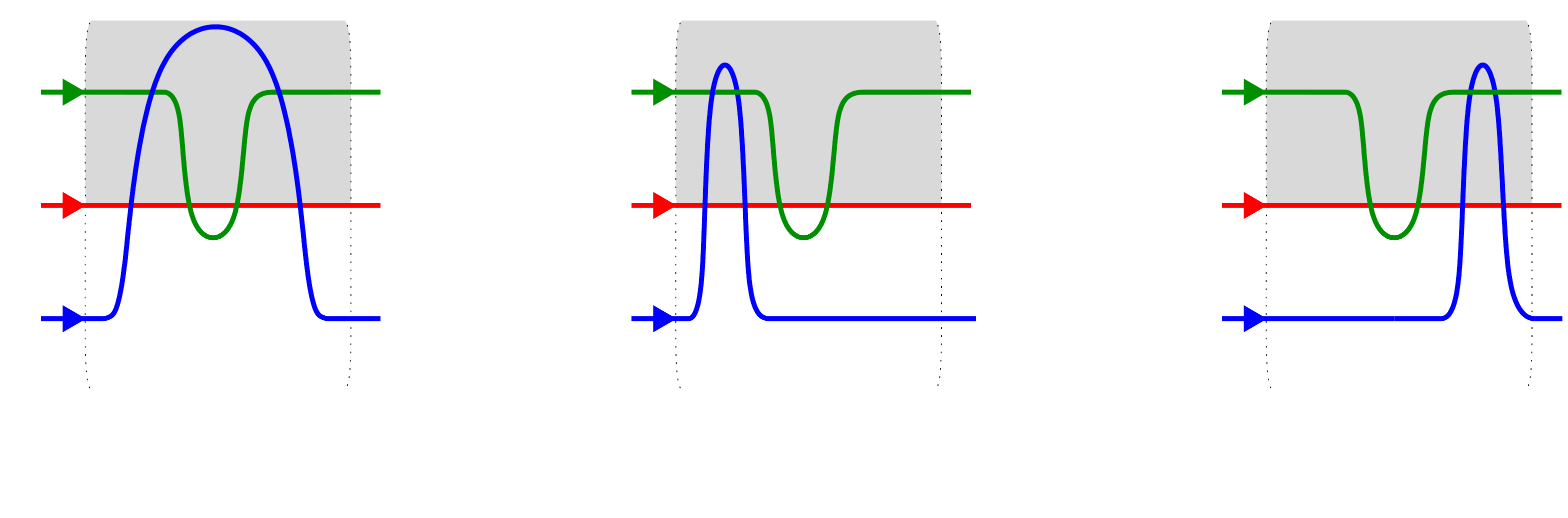}
\caption{Illustration of the proof of Proposition~\ref{pro:claimb}.}
\label{fig:figW}
\end{figure}

Regard $\qq$ as a complete $3${-}uniform hypergraph. We assign to a $3${-}edge $\{j,k,\ell\}$ with $j{<}k{<}\ell$ a colour in $\{1,2,3\}$, depending, respectively, on which one of (1), (2), and (3) gives the code of $k$ in $\{j,k,\ell\}$. Since $q=r_3(m,m,m)$, by Ramsey's theorem $\qq$ has a subarrangement $\mm=\{1',2',\ldots,m'\}$, where $1'{<}2'{<}\cdots{<}m'$, all of whose $3$-edges are of the same colour. If all the $3$-edges are of colour $3$, then by changing the traversal directions of all pseudocircles and inverting the crossing-sign convention we obtain that all the $3$-edges become of colour $2$. Thus we may assume that either all the $3$-edges are of colour $1$, or they are all of colour 2. Again, to avoid unnecessary cluttered notation we relabel the pseudocircles in $\mm$ with $1,2,\ldots,m$.

Suppose first that all the $3$-edges of $\mm$ are of colour 1. Then for each $1{\le} j{<}k{<}\ell{\le} m$, the situation is as in Figure~\ref{fig:figW}(a). Thus (i) the code of $j$ in $\{j,k,\ell\}$ is $\menos{k}\menos{\ell}\mas{\ell}\mas{k}$; (ii) the code of $k$ in $\{j,k,\ell\}$ is $\mas{j}\menos{\ell}\mas{\ell}\menos{j}$; and (iii) the code of $\ell$ in $\{j,k,\ell\}$ is $\mas{j}\mas{k} \menos{k}\menos{j}$. We first analyze the code of each pseudocircle $i\in\{2,3,\ldots,m{-}1\}$. By (i), the code of $i$ in $\mm$ contains the subpermutation $(\menos{i}{:}\menos{m}][\mas{m}{:}\mas{i})$; by (iii), this code contains the subpermutation $[\mas{1}{:}\mas{i})(\menos{i}{:}\menos{1}]$; and by (ii), it contains the subpermutation $\mas{(i-1)}\menos{(i+1)}\mas{(i+1)}$ $\menos{(i-1)}$. These three conditions hold only if the code of $i$ is $[\mas{1}{:}\mas{i})(\menos{i}{:}\menos{m}][\mas{m}{:}\mas{i})$ $(\menos{i}{:}\menos{1}]$. 

We now analyze the codes of pseudocircles $1$ and $m$. By (i), the code of $1$ in $\mm$ is $(\menos{1}{:}\menos{m}][\mas{m}{:}\mas{1})= [\mas{1}{:}\mas{1}) (\menos{1}{:}\menos{m}][\mas{m}{:}\mas{1}) (\menos{1}{:}\menos{1}]$ (since $[\mas{1}{:}\mas{1})$ and $(\menos{1}{:}\menos{1}]$ are empty strings). Finally, by (iii), the code of $m$ in $\mm$ is $[\mas{1}{:}\mas{m})(\menos{m}{:}\menos{1}] =    [\mas{1}{:}\mas{m})$ $(\menos{m}{:}\menos{m}]$ $[\mas{m}{:}\mas{m})$ $ (\menos{m}{:}\menos{1}]$ (since  $(\menos{m}{:}\menos{m}]$ and $[\mas{m}{:}{m})$ are empty strings). We conclude that the code of every $i\in\mm$ is $[\mas{1}{:}\mas{i})(\menos{i}{:}\menos{m}]$ $[\mas{m}{:}\mas{i})(\menos{i}{:}\menos{1}]$. By Proposition~\ref{pro:ortner} and Observation~\ref{obs:obsctwo}, then $\mm$ is isomorphic to $\ctwom$.

Suppose finally that all the $3$-edges of $\mm$ have colour 2. Then for each $1{\le} j{<}k{<}\ell{\le} m$, the situation is as in Figure~\ref{fig:figW}(b). Thus (i) the code of $j$ in $\{j,k,\ell\}$ is $\menos{k}\menos{\ell}\mas{\ell}\mas{k}$; (ii) the code of $k$ in $\{j,k,\ell\}$ is $\mas{j}\menos{j}\menos{\ell}\mas\ell$; and (iii) the code of $\ell$ in $\{j,k,\ell\}$ is $\mas{j}\menos{j} \mas{k}\menos{k}$. We first analyze the code of each pseudocircle $i\in\{2,3,\ldots,m{-}1\}$. By (i), the code of $i$ in $\mm$ contains the subpermutation $(\menos{i}{:}\menos{m}][\mas{m}{:}\mas{i})$; by (iii), this code contains the subpermutation $[\mas{1}\menos{1}{:}\mas{i}\menos{i})$; and by (ii), it contains the subpermutation $\mas{(i-1)}\menos{(i-1)}\menos{(i+1)}$ $\mas{(i+1)}$. These three conditions hold only if the code of $i$ is $[\mas{1}\menos{1}{{:}}\mas{i}\menos{i})$ $(\menos{i}{{:}}\menos{m}][\mas{m}{{:}}\mas{i})$. 

We now analyze the codes of $1$ and $m$. By (i), the code of $1$ in $\mm$ is $(\menos{1}{{:}}\menos{m}][\mas{m}{:}\mas{1}) =   [\mas{1}\menos{1}{:}\mas{1}\menos{1}) (\menos{1}{{:}}\menos{m}][\mas{m}{:}\mas{1})$ (since $[\mas{1}\menos{1}{:}\mas{1}\menos{1})$ is an empty string). Finally, by (iii), the code of $m$ in $\mm$ is $[\mas{1}\menos{1}{:}\mas{m}\menos{m}) = [\mas{1}\menos{1}{:}\mas{m}\menos{m}) (\menos{m}{:}\menos{m}] [\mas{m} {:}\mas{m} )$ (since $(\menos{m}{:}\menos{m}]$ and $[\mas{m} {:}\mas{m})$ are empty strings). We conclude that the code of every $i\in\mm$ is $[\mas{1}\menos{1}{{:}}\mas{i}\menos{i})(\menos{i}{{:}}\menos{m}]$ $[\mas{m}{{:}}\mas{i})$. By Proposition~\ref{pro:ortner} and Observation~\ref{obs:obscthree}, then $\mm$ is isomorphic to $\cthreem$.
\end{proof}

{
\section{An open question}
To prove Theorem~\ref{thm:main} we make repeated use of Ramsey's theorem, and so an explicit bound for this theorem, derived from our proofs, would be multiply exponential. With additional effort (and considerably more space) we can save several applications of Ramsey's theorem, and show that for each fixed $m\ge 1$, every arrangement of pseudocircles of size at least $2^{2^{cm^2}}$ contains a subarrangement isomorphic to $\conem,\ctwom$, or $\cthreem$. This bound is still doubly exponential in $m$. What is the best explicit bound that can be proved for this theorem?
}


\section*{Acknowledgements}

We thank Stefan Felsner and Manfred Scheucher for making available to us their software to generate all arrangements of pseudocircles of small order; experimenting with this code was very useful at the early stages of this project. The first author is supported by Fordecyt grant 265667. The second author is partially supported by PICS07848 grant. The third author is supported by Conacyt grant 222667 and by FRC-UASLP.

\bibliographystyle{abbrv}
\bibliography{refs.bib}

\end{document}